\documentclass[10pt,a4paper]{article}
\usepackage[utf8]{inputenc}
\usepackage{amsmath}
\usepackage{amsfonts}
\usepackage{amssymb}
\usepackage{amsmath}
\usepackage{amsthm}
\usepackage{amssymb}
\usepackage{enumerate}	
\usepackage{xcolor}
\newtheorem{thm}{Theorem}[section]
\newtheorem{theorem}{Theorem}
\newtheorem{cor}[thm]{Corollary}
\newtheorem{lm}[thm]{Lemma}
\newtheorem{pr}[thm]{Proposition}
\theoremstyle{definition}
\newtheorem{df}[thm]{Definition}
\theoremstyle{remark}
\newtheorem{rem}[thm]{Remark}
\usepackage[a4paper, left = 2.5cm, right = 2.5cm, top = 2.5cm, bottom = 2.5cm, headsep = 1.2cm]{geometry}
\author{Karol Duda and Aleksander Ivanov}
\title{Computable F\o lner sequences of amenable groups}
\date{}
\begin{document}
\maketitle

\begin{abstract}

The paper considers computable F\o lner sequences in  computably enumerable amenable groups.  
We extend some basic results of M. Cavaleri on existence of such sequences to the case of groups where finite generation is not assumed. 
We also initiate some new directions in this topic, for example complexity of families of effective F\o lner sequences.
Possible extensions of this approach to metric groups are also discussed. 
\end{abstract}

\section{Introduction} 
Analysis of classical mathematical topics  from the point of view of complexity of various types is one of the major trends of modern mathematical logic. 
Amenability is essentially fruitful from this point of view
(see \cite{BK},  \cite{CK},\cite{HPP}, \cite{HKP1}, \cite{HKP2},  \cite{kechrisN}, \cite{KPT},
\cite{MU} ).
Our research belongs to computable amenability. 
This is a topic where computable versions of fundamentals of amenability are studied, see papers of M. Cavaleri \cite{MC2, MC3}, N. Moryakov \cite{mor} and the authors \cite{DuI, CPD}. 
Let us also mention \cite{sim} initiating a very rich field where computability meets topological dynamics. 

In the present paper we return to the results of M. Cavaleri from \cite{MC3}, which are now considered as the beginning of the topic. 
It has been shown in \cite{MC3} that amenable finitely generated recursively presented groups 
have computable Reiter functions and subrecursive F\o lner functions.
Furthermore, for such a group decidability of the word problem is equivalent to so called {\em effective amenability}, i.e.  
existence of an algorithm which finds $\frac{1}{n}$-F\o lner sets for all $n$.

Since being finitely generated is not necessary for amenability, the question arises what happens 
if we consider the case of recursively presented groups without the assumption of finite generation.
According to the approach of computable algebra, the question concerns the class of {\em computably enumerable numbered groups} and 
the subclass of {\em computable numbered groups}, a counterpart of decidability of the word problem.
These notions are thoroughly discribed in Section 2. 
The following theorem generalizes aforementioned results of Cavaleri to the case of computably enumerable numbered groups. 

\begin{theorem} \label{1} 
Let $(G,\nu )$ be a computably enumerable numbered group.
The following conditions are equivalent:
\begin{enumerate}[(i)]
\item $G$ is amenable;
\item $(G, \nu )$ has computable Reiter functions;
\item $(G, \nu )$ has subrecursive F\o lner function.
\item $(G, \nu )$ is $\Sigma$-amenable (see Definition \ref{cea}).
\end{enumerate}
Furthermore, computable amenability of $(G, \nu )$ is equivalent to  computability of it.
\end{theorem} 
\noindent 
This theorem summarizes our results of Section 3. 

In the second part of the paper (Section 4) we concentrate on algorithmic complexity of effective F\o lner sequences 
and families of these sequences.  
In particular, we prove the following theorem. 

\begin{theorem} \label{2} 
The the set of all effective F\o lner sequences of a computable group belongs to the class $\Pi^0_3$, and, furthermore,
in some cases of abelian groups this family is $\Pi^0_3$-complete.  
\end{theorem} 
\noindent 
We also compare convergence moduli of sequences of means corresponding to these F\o lner sequences. 
In particular we show that in the case of the standard F\o lner sequence of $(\mathbb{Z},+)$ and 
the corresponding sequence of means $m_i ({\bf x})$ (which converge to an invariant mean witnessing amenability of $\mathbb{Z}$) 
the following statement holds. 

\begin{theorem} \label{2_5}
For any total computable $f: \mathbb{N} \to \mathbb{N}$ there is a computable ${\bf x}_0 \in 2^{\mathbb{Z}}$ such that the seqience $m_i ({\bf x}_0)$, $i\in \mathbb{N}$, converges to $0$, but for every $k\in \mathbb{N}$ there is $j > f(k)$ such that $| m_j ({\bf x}_0) | \ge \frac{1}{k}$.  
\end{theorem}

In the final part of our paper (Section 5)  we study possible generalizations of our results to computable metric groups.  
We suggest a framework to computable amenability in this general case. 
In particular, we define and discuss counterparts of basic notions studied in the main body of the paper. 
In these terms we prove the following theorem.  

\begin{theorem} \label{3} 
A computably enumerable numbered metric group $(G,d,\nu )$ is computably amenable if and only if it is amenable and computable. 
\end{theorem} 
\noindent 
This is an extension of the final statement of Theorem \ref{1} to the case of metric groups.  

The authors are grateful to M. Cavaleri and T. Ceccherini-Silberstein for reading 
of some preliminary versions of the paper and helpful remarks.

\section{Preliminaries} 
In this section we give preliminaries on (a) computable groups and (b) amenable groups. 
While (b) is rather standard, the topic (a) is presented in the form which seems new. 
We have found that the notion of enumerated groups from \cite{GKEL} allows us to simplify the presentation of the material and some arguments. 

Concerning other details, let us mention that we often identify finite sets $F\subset \mathbb{N}$ with their  G\" odel numbers.
For any sets $X$ and $Y$ we will write $X \subset_{fin} Y$ to denote that $X$ is a finite subset of $Y$.  
Throughout this paper a group (metric group) $G$ is a countable/separable  group 
without any presumption about its generating set. 

A function is {\em subrecursive} if it admits a computable total upper bound.
A sequence $(n_i)_{i\in\mathbb{N}}$ of natural numbers is called {\em computable/effective}, 
if the function $k\rightarrow n_k$ is recursive.
We use standard material from the computability theory (see  \cite{sri}) and often say computable instead of recursive.

\subsection{Computabie presentations}

Let $G$ be a countable group generated by some $X\subseteq G$. 
The group $G$ is called {\em recursively presented} (see Section IV.3 in \cite{ls}) if 
$X$ can be identified with $\mathbb{N}$ (or with some $\{ 0,\ldots , n\}$) 
so that $G$ has a recursively enumerable set of relators in $X$.  
Below we give an equivalent definition, see Definition \ref{df1}. 
It is justified by a possibility identification of the whole $G$ with $\mathbb{N}$. 
We develope the approaches of \cite{EG}, \cite{EG1} and \cite{khmi}.

\begin{df}\label{df0}
Let $G$ be a group and $\nu: \mathbb{N} \rightarrow G$ be a surjective function. 
We call the pair $(G,\nu)$ a {\em numbered group}.
The function $\nu$ is called a {\em numbering} of $G$.
If $g\in G$ and $\nu(n)=g$, then $n$ is called a number of $g$.
\end{df} 

In this paper we usually assume that numberings of the group are homomorphisms 
from some group defined on $\mathbb{N}$. 
This condition is formulated in the following definition. 
The notion of enumerated group used in it, is taken from \cite{GKEL}. 

\begin{df} \label{df1} 
\begin{itemize} 
\item A group of the form $(\mathbb{N}, \star , ^{-1}, 1 )$ (where the number $1$ is the neutral element of the group) is called an {\em  enumerated group}. 
\item Given an enumerated group $(\mathbb{N}, \star , ^{-1}, 1 )$ we call a surjective homomorphism $\nu : \mathbb{N} \to G$ a {\em  computably enumerable presentation} of $G$ if the set 
$$
\mathsf{Wrd}^=_{\nu} := \{(w(n_1  ,\ldots ,n_s), w'(\ell_1 , \ldots ,\ell_t ))\, | \, w(\bar{x}) \mbox{ and }  w'(\bar{y}) \mbox{ are group words and the equality } 
$$ 
$$ 
w(\nu (n_1 ),\ldots ,\nu (n_s)) = w'(\nu (\ell_1 ) , \ldots , \nu (\ell_t )) \mbox{ holds in } G \,  \, , \, n_1  ,\ldots ,n_s , \ell_1 , \ldots ,\ell_t \in \mathbb{N}\}
$$ 
is computably enumerable.
\end{itemize} 
\end{df} 
An easy folklore argument shows that every finitely generated group with decidable word problem 
can be presented as an enumerated group $(\mathbb{N} , \star , ^{-1}, 1)$ such that $\star$ and the corresponding $^{-1}$ are computable functions.  
This also holds in the case of the free group $\mathbb{F}_{\omega}$ 
with the free basis $\omega = \{ 0 , \ldots , i , \ldots \}$. 
From now on let us fix such a presentation of $\mathbb{F}_{\omega}$: 
$$ 
(\mathbb{N}, * , ^{-1}, 1). 
\footnote{note that for multiplication here we use $*$, which is different from $\star$}
$$
We emphasize that this group is {\em computable}, i.e. its operations are computable functions!   
For every recursively presented group $G= \langle X\rangle$ 
and a natural homomorphism $\rho : \mathbb{F}_{\omega} \rightarrow G$ 
(taking $\omega$ onto $X$) we obtain that $\rho$ is a numbering 
of a computably enumerable presentation of $G$. 
This argument can be generalized as follows.

\begin{lm} \label{comp_gr} 
Assume that $(G, \nu )$ is a numbered group and the set $\mathsf{Wrd}^=_{\nu}$
is computably enumerable. 

Then $G$ has a computably enumerable presentation 
$\nu^0 : (\mathbb{N} , *, ^{-1}, 1) \to G$.  
\end{lm} 

{\em Proof.} 
Note that extending  the map $\nu$ to the set of all group words over the base 
$\omega = \{ 0, 1 , \ldots , i , \ldots \}$ 
we obtain a homomorphism $\rho: \mathsf{F}_{\omega} \to G$ such that the set 
$$
\mathsf{Wrd}^=_{\rho} := \{(w(u_1  ,\ldots ,u_s), w'(v_1 , \ldots ,v_t ))\, | \, w(\bar{x}) \mbox{ and }  w'(\bar{y}) \mbox{ are group words, } 
$$ 
$$ 
\bar{u}, \bar{v} \in \mathbb{F}_{\omega} \mbox{ and the equation } w(\rho (u_1 ),\ldots ,\rho (u_s)) = w'(\rho (v_1 ) , \ldots , \rho (v_t )) \mbox{ holds in } G \, \}
$$ 
is computably enumerable. 

Under the coding, which we mentioned above, of the group $\mathbb{F}_{\omega}$ into the enumerated group $(\mathbb{N} , * , ^{-1}, 1)$ the homomorphism $\rho$  becomes the numbering $\nu^0 : (\mathbb{N} , *, ^{-1}, 1) \to G$ as in the formulation. 
$\Box$ 

\bigskip 

In particular, every group with a computably enumerable presentation has a presentation as in the lemma, i.e. where the enumerated group is computable. 
It is convenient to use the following notions too. 

\bigskip

\begin{df} 
\begin{itemize} 
\item If $G$ has a computably enumerable presentation then we say that $G$ is {\em computably enumerable}. 
\item If a homomorphism $\nu : \mathbb{N} \to G$ is a computably enumerable presentation of $G$ and  the set 
$\mathsf{Wrd}^=_{\nu}$
is computable (i.e. decidable), then we say that $\nu$ is a {\em computable presentation} and the group $G$ is {\em computable}.
\end{itemize} 
\end{df} 

Note that under the conditions of Lemma \ref{comp_gr} the sets $\{n:\nu^0 (n)=1\}$ and $\{(n_1,n_2): \nu^0 (n_1)=\nu^0 (n_2)\}$ are computably enumerable.
In the following remark we consider the case when they are computable.

\begin{rem}\label{rp}
Let $\nu : (\mathbb{N} , \star, ^{-1}, 1) \to G$ be a computably enumerable presentation such that $\star$ and the corresponding $^{-1}$ are computable functions. 
\begin{enumerate}[(i)] 
\item  
The set 
$$
\mathsf{MultT} := \{(i,j,k): \quad \nu (i)\nu (j)=\nu (k)\}
$$  
is computable if and only if the set $\{(n_1,n_2): \nu (n_1)=\nu (n_2)\}$ is computable; 
\item If $\mathsf{MultT} $ is computable, the presentation $\nu$ is computable too and, furthermore, it can be made a computable bijective presentation. 
\end{enumerate} 
Indeed, in this case the set of the smallest numbers of the elements of $G$ is computable. 
Enumerating this set by natural numbers we obtain 
a required bijective numbering. 
\end{rem} 
Computable groups correspond to groups with solvable word problem.
In this case the numbering $\nu$ is often called a {\em constructivization}, see \cite{EG}, \cite{EG1}. 
 
From now on we consider computably enumerable groups in the following way. 
\begin{itemize} 
\item Any computably enumerable group $(G, \nu )$ is taken with a homomorphis  
$\nu : (\mathbb{N}, \star ,^{-1} , 1) \to G$ where $(\mathbb{N}, \star , ^{-1},1)$ is a group with computable operations, i.e. satisfying the conditions of Lemma \ref{comp_gr}. 
\item If $(G, \nu )$ is computable, we additionally assume that $\nu$ is an isomorphism. 
\end{itemize}

\subsection{Amenability}
The preliminaries of amenability correspond to \cite{csc} and \cite{pat}. 
Let $G$ be a group, and $\ell^{\infty} (G)$ be the Banach space of bounded functions $G \to \mathbb{R}$ 
with respect to the the norm 
\[ 
\parallel {\bf x} \parallel_{\infty} = \mathsf{sup} \{ |{\bf x}(a)| \, | \, a \in G \}. 
\] 
The group $G$ is called {\em amenable} if there is a left-invariant (equivalently rigth-invariant, resp. bi-invariant) mean 
$\ell^{\infty} (G) \to \mathbb{R}$ on $G$.  
\begin{df} 
Given $n\in \mathbb{N}$ and $D\subset_{fin} G$, a subset $F\subset_{fin} G$ is called 
an $\frac{1}{n}$-{\em F\o lner set} with respect to $D$ if 
\begin{equation}
\forall x\in D\hspace{0.5cm}  \frac{|F\setminus xF|}{|F|} \leq \frac{1}{n}. 
\end{equation} 
\end{df}    

\noindent 
We denote by $\mathfrak{F}$\o$l_{G, D}(n)$ the set of all $\frac{1}{n}$-F\o lner sets with respect to $D$.
Moreover, we call the binary function:
\begin{equation}
F\o l_{G}(n,D) = \mathsf{min}\{|F| \, | \, F\subseteq G \text{ such that } F \in \mathfrak{F} \o l_{G, D}(n)\},
\end{equation}
\newline
where the variable $D$ corresponds to finite sets, the {\em F\o lner function of} $G$. 

A sequence $(F_j)_{j\in \mathbb{N}}$ 
of non-empty finite subsets of $G$ is a (left) {\em F\o lner sequence}, 
if for every $g \in G$ the following condition holds:
\begin{equation}
 \lim\limits_{j\rightarrow \infty}\frac{|F_j\setminus gF_j|}{|F_j|}=0.
\end{equation} 
It is easy to see that existence of F\o lner sets for all $n$ and $D$  
is equivalent to existence of a F\o lner sequence: 

\bigskip 
$\bullet$ $G$ admits a F\o lner sequence if and only if 
$F\o l_{G}(n,D)< \infty$ for all  $n\in\mathbb{N}$ and $D\subset_{fin} G$. 

\bigskip 
\noindent 
In fact, this is the {\em F\o lner condition of amenability}.

The following example will be helpful below. 
The group $(\mathbb{Z}, +)$ has the following  F\o lner sequence: 
$$
\mathcal{F} = (\{ - i , -i +1 , \ldots , 0 , \ldots , i-1 , i \} \, | \, i \in  \mathbb{N} ). 
$$    
Note that $\{ - i , -i +1 , \ldots , 0 , \ldots , i-1 , i \}$ is $\frac{1}{2i}$-F\o lner with respect to the generator of $\mathbb{Z}$.   

Suppose that $G$ satisfies the F\o lner conditions. 
Let $( F_j  \, | \,  j\in J )$ be a  left F\o lner sequence of $G$. 
Consider, for each $j\in J$, the mean with finite
support $m_j : \ell^{\infty} (G) \to \mathbb{R}$ defined by
$$
m_j ( {\bf x} ) = \frac{1}{| F_j |} \sum \{   {\bf x}  (h) \,  | \, h \in F_j \} 
 \mbox{ for all } {\bf x} \in \ell^{\infty} (G) . 
$$  
The net $( m_j \, | \, j\in J)$ contains a sequence converging (in the weak$^*$ topology) to an invariant mean $m$ 
(see Lemma 4.5.9 and Theorem 4.9.2 in \cite{csc}). 

The space of absolutely summable functions $\ell^1 (G)$ is considered with respect to the norm 
\[ 
\parallel {\bf x} \parallel_1 = \sum \{ | {\bf x}(a) | \, | \, a\in G \} \mbox{ , where } {\bf x} : G\to \mathbb{R} \mbox{ is absolutely summable.} 
\] 

\begin{df}
A non-zero function $h: G \rightarrow \mathbb{R}^{+}$, $||h||_{1}<\infty$, is $\frac{1}{n}$-{\em invariant} with respect to $D$, if 
\begin{equation}\label{rse}
\forall x\in D\hspace{0.5cm} \frac{||h - _{x}h||_{1}}{||h||_{1}}< \frac{1}{n},
\end{equation}
where $_{x}h(g):=h(x^{-1}g)$.
\newline
We denote by $\mathfrak{R}eit_{G, D}(n)$ the set of all summable non-zero functions from $G$ to $\mathbb{R}^{+}$, 
which are $\frac{1}{n}$-invariant with respect to $D$. 
\end{df}

\noindent The following facts are well known and/or easy to prove (see also Remark 2.2 from \cite{MC3}). 

\begin{lm}\label{fr} Let $F,D\subset_{fin} G$.

\begin{enumerate}[(i)]
\item $F \in \mathfrak{F} \o l_{G, D}(n) \implies \forall g \in G \quad Fg \in F \o l_{G, D}(n)$

\item $F \in \mathfrak{F} \o l_{G, D}(n) \iff \forall x \in D  \quad \frac{|F\cap xF|}{|F|}> 1 - \frac{1}{n}$

\item $F \in \mathfrak{F} \o l_{G, D}(2n) \iff \chi_F \in \mathfrak{R}eit_{G, D}(n)$

\item If $h\in \mathfrak{R}eit_{G,D}(n)$ has finite support then there exists $F \subset \mathsf{supp}(h)$ such that for all $x \in D$ the following holds: 
\[ 
\frac{|F\setminus xF|}{|F|}< \frac{|D|}{2n}. 
\]

\end{enumerate}
\end{lm}

Let $Prob (G)$  be the set of countably supported probability measures on $G$ (with respect to the $\sigma$-algebra of all subsets of 
$G$). 
Under the assumption (in Sections 3 and 4) that $G$  is countable, $Prob(G)$ is exactly the set of probability measures on $G$. 
Any element of this space  can be written $\mu = \sum_{i\in \mathbb{N}} \lambda_i \chi_{g_i}$  where $g_i \in G$, $\lambda_i \ge 0$ and 
$\sum_{i\in \mathbb{N}} \lambda_i = 1$. In particular, it can identified with a subset of the unit ball of $\ell^1 (G)$.

Reiter's condition states that the group is amenable if and only if for any finite subset $D\subset G$  and $\varepsilon > 0$, 
there is  $\mu \in Prob(G)$  such that 
$\parallel \mu - _{g}\mu \parallel_1 \le \varepsilon$  for any $g\in D$. 
This justifies the  term $\mathfrak{R}eit_{G, D}(n)$.

\section{Effective amenability of computably enumerable groups}

In this section $(G,\nu)$ is a numbered metric group such that $\nu$ is a homomorphism from $(\mathbb{N},\star ,^{-1},1)$ to $G$ 
where operations $\star$ and $^{-1} $ are computable. 
In the case of F\o lner's condition of amenability, we consider two types of effectiveness.
\begin{df}\label{cea} 
The numbered group $(G,\nu )$ is $\Sigma$-{\em amenable}, if there is an algorithm which 
for all pairs $(n,D)$ where $n\in \mathbb{N}$ and 
$D\subset_{fin}\mathbb{N}$, finds a set $F\subset_{fin}\mathbb{N}$ having a subset $F' \subseteq F$ with 
$\nu (F')\in \mathfrak{F} \o l_{G,\nu(D)}(n)$. 
\end{df} 

\begin{df}\label{ca} 
The numbered  group $(G,\nu)$ is {\em computably amenable} if
there exists an algorithm which for all pairs $(n,D)$, 
where $n\in \mathbb{N}$ and $D\subset_{fin}\mathbb{N}$, finds a finite set $F\subset\mathbb{N}$ such that 
$\nu(F) \in \mathfrak{F} \o l_{ G,\nu(D)}(n)$ and $|F|=|\nu(F)|$.
\end{df}

These are the main notions of our paper. 
It is clear that computable amenability implies $\Sigma$-amenability. 
Furthermore, $\Sigma$-amenability implies that the F\o lner function is subrecursive. 

Some variants of these notions were characterized by M. Cavaleri in \cite{MC3} in the case of finitely generated groups. 
The goal of this section is an adaptation of these characterizations in our general case.   
Although we use the same arguments, the adaption needs some additional effort.

\subsection{Computable Reiter's functions} 

The main result of this part, Theorem \ref{re}, is a natural generalization of a theorem of 
M. Cavaleri from \cite{MC3} (Theorem 3.1) to the case of groups which are not finitely generated. 
Throughout this section we assume that $(G,\nu)$ is a computably enumerable group under a homomrphism $\nu$ as in the beginning of the section.  

The following notation will be used below. 
It corresponds to Section 3 of \cite{MC3}. 
If $f: \mathbb{N} \to \mathbb{R}^+$ is summable and $g\in G$, then  let 
\[ 
\nu_{G,1}(f)(g)=\sum\{ f(i) \, | \, i\in\nu^{-1}(g) \}. 
\] 

\begin{df}\label{cr} 
We say that $(G,\nu)$ has {\em computable Reiter functions}, if there exists an algorithm which, for every $n\in \mathbb{N}$ and any finite set $D\subset\mathbb{N}$ finds $f: \mathbb{N} \rightarrow \mathbb{Q^{+}}$, such that $|\mathsf{supp}(f)| < \infty$ and 
$$
\forall x \in D, \quad \frac{||\nu_{G ,1}(f) - _{\nu (x)}\nu_{G ,1}(f)||_{1}}{||\nu_{G ,1}(f)||_{1}}< \frac{1}{n}, 
$$ 
\end{df}

We now need  some preliminary material concerning Reiter functions and partitions. 
Let $X$ be a nonempty set. 
An equivalence relation $E'$ on $X$ is called finer than an equivalence relation $E$ if $E' \subseteq E$.

Let $f:\mathbb{N}\rightarrow\mathbb{Q}_{+}$ be a function with a finite support $F$ and let $X$ be a finite set including $F$. 
For an equivalence relation $E$ of $X$ define $E_{F,1} (f)$ as follows:  
\[ 
E_{F,1}(f)(x):=\sum\{ f(i) \, | \, i\in F \mbox{ and } (i,x) \in E \}. 
\] 
Having $x\in \mathbb{N}$ and an equivalence relation $E$ on the set $F \cup x^{-1}\star F$ fix a set $F_0 \subset F$ of representatives of the classes of $E$ in $F$.  
Then we define the positive rational number:
$$
M_{E,F}^{x}(f):= \frac{\sum \{ | (E_{F,1}(f)(v)-E_{F,1}(f) (x^{-1}\star v))|  \, | \, v\in F_0 \} }{\sum_{v\in F}f(v)} .  
$$ 
We denote by  $P$ the canonical equivalence relation of the set $F\cup x^{-1} \star F$, 
i.e. the partition into sets $ \{\nu^{-1}(\nu(k))  \, | \, k\in F\cup x^{-1}\star F\}$. 
Then for every $x\in \mathbb{N}$ we have 
\begin{equation}\label{5}
M_{P,F}^{x}(f) = \frac{||\nu_{G .1}(f)-_{\nu(x)}\nu_{G ,1}(f)||_{1}}{||\nu_{G ,1}(f)||_{1}}.
\end{equation}
For any two equivalence relations $E$ and $E'$ of the set $F\cup x^{-1}\star F$ with $E\subseteq E'$, the triangle inequality implies 
$M_{E,F}^{x}(f) \geq M_{E',F}^{x}(f)$. 
In particular, for any  $x\in \mathbb{N}$ and $E\subseteq  P$ on $F \cup x^{-1}\star F$ we have: 
\begin{equation}\label{6}
M_{E,F}^{x}(f) \geq M_{P,F}^{x}(f).
\end{equation}

\begin{lm}\label{rr} 
Let $(G,\nu)$ be a computably enumerable group as above. 
There exists a computable enumeration of the set of all triples $(n, D, f)$, where 
$D \subset_{fin} \mathbb{N}$ and $f:\mathbb{N}\rightarrow\mathbb{Q}^{+}$ is a finitely supported function, 
such that $\nu_{G ,1}(f) \in \mathfrak{R}eit_{G,\nu(D)}(n)$. 
\end{lm}

\begin{proof}
We apply the method of Theorem 3.1($(i)\rightarrow (iv)$) of \cite{MC3}. 
Let us fix an enumeration of functions $f$ with finite support as in the formulation, and 
the corresponding enumeration of all triples of the form $(n,D,f)$: $(n_1 ,D_1 ,f_1)$, $(n_2 ,D_2 ,f_2), \ldots$ , $(n_k,D_k ,f_k), \ldots$.
Let us also fix an enumeration of the set $\{(n_1,n_2): \nu(n_1)=\nu(n_2)\}$.
The following procedure, denoted below by $\varkappa(n,D,f)$, verifies if the triple satisfies the condition of the lemma.

The algorithm $\varkappa(n,D,f)$ starts as follows. 
For an input $f$ let $F=\mathsf{supp} (f)$. 
Put $P_0$ to be  the (finest) partition of $\bigcup_{g\in D} g^{-1} \star F$ into singletons. 
At the $m$-th step of the enumeration of $\{(n_1,n_2): \nu(n_1)=\nu(n_2)\}$ 
we are trying to merge classes of the equivalence relation $P_{m-1}$ on $\bigcup_{g\in D} g^{-1}\star F$ already constructed at step $m-1$. 
We do so when we meet $(n_1,n_2)\not\in P_{m-1}$ with $\nu (n_1 ) = \nu (n_2 )$. 
In this case we just merge the classes of $n_1$ and $n_2$. 
In this way we obtain $P_m \subseteq P$. 
Then we verify if $M_{P_{m},F}^{x}(f) \leq \frac{1}{n}$ for all $x\in D$. 
We stop $\varkappa (n,D,f)$ when these inequalities hold. 
In this case  by (\ref{5}) and (\ref{6}), the function $\nu_{G ,1}(f)$ is $\frac{1}{n}$-invariant with respect to $D$.
If there exist $x$, such that $M_{P_{m},F}^{x}(f) > \frac{1}{n}$ and $P_m = P$, 
 then the function $\nu_{G ,1}(f)$ is not $\frac{1}{n}$-invariant.  
Note that there is no algorithm for recognizing the latter possibility. 

The algorithm stated in the formulation of the lemma at $k$-th step makes the first move 
in  $\varkappa (n_k,D_k,f_k )$, the second move in $(n_{k-1},D_{k-1},f_{k-1}), \ldots $, and the $k$-th move in  $\varkappa (n_1,D_1,f_1 )$. 
When one of these procedures gives 
$\nu_{G ,1}(f) \in \mathfrak{R}eit_{G,\nu(D)}(n)$ we put the corresponding triple into our list. 
\end{proof}

The following theorem is a part of Therem \ref{1} from the introduction. 
The proof uses the procedure $\varkappa (n,D,f)$ from the proof of Lemma \ref{rr}. 

\begin{thm}\label{re}
Let $(G,\nu)$ be a computably enumerable group. 
Then the following conditions are equivalent:

\begin{enumerate}[(i)]
\item $G$ is amenable; 

\item $(G,\nu)$ has a subrecursive F\o lner function;

\item $(G,\nu)$ is $\Sigma$-amenable;

\item $(G,\nu)$ has  computable Reiter functions.

\end{enumerate}
\end{thm}

\begin{proof}
It is clear that 
(iii)$\implies$(ii)$\implies$(i).  

(iv)$\implies$(iii). 
By Definition \ref{cr} there is an algorithm which for every $n\in \mathbb{N}$ and every $D \subset_{fin} \mathbb{N}$ finds a function 
$f: \mathbb{N} \rightarrow \mathbb{Q^{+}}, |\mathsf{supp}(f)| < \infty$, such that 
$\nu_{G ,1}(f)\in \mathfrak{R}eit_{G, \nu (D)}(n)$. 
Denote $F = \mathsf{supp}(f)$. 
By Lemma \ref{fr} $(iv)$, 
there exists  $F'\subseteq F$ such that $\nu(F')$ satisfies F\o lner's condition with respect to $\nu (D)$. 

To prove (i)$\implies$(iv) let us assume that the group $G$ is amenable. 
Therefore, for any $n$ and $D\subset_{fin} \mathbb{N}$ there exists $F\subset_{fin} \mathbb{N}$ such that 
$\nu(F) \in \mathfrak{F}\o l_{G, \nu(D)}(2n)$ and $|F| = |\nu(F)|$. 
Since $\nu$ is injective on $F$, we see by Lemma \ref{fr} (iii) that $\nu_{G ,1}(\chi_F)=\chi_{\nu(F)}\in \mathfrak{R}eit_{G, \nu(D)}(n)$. 
Now fix an enumeration of finite subsets of $\mathbb{N}: F_1,F_2,\ldots$ and start the algorithms 
$\varkappa(n,D,\chi_{F_1}), \varkappa(n,D,\chi_{F_2}),\ldots$ constructed in Lemma \ref{rr}, 
until one of them stops giving us a Reiter function for $\nu(D)$. 
\end{proof}

\subsection{Effective amenability of computable groups}

The main results of this section, correspond to Theorem 4.1 and Corollary 4.2 of M. Cavaleri  from \cite{MC3}. 
In the proof we will use functions $*$ and $^{-1}$ from Lemma \ref{comp_gr}. 
In particular, we are under the conditions of  Remark \ref{rp}. 

\begin{thm}\label{ce}
Let $(G,\nu)$ be a computably enumerable group. The following conditions are equivalent:

\begin{enumerate}[(i)]
\item $(G,\nu)$ is amenable and computable; 

\item $(G,\nu)$ is computably amenable (Definition \ref{ca}).

\end{enumerate}
\end{thm}

\begin{proof}
(i)$\implies$(ii).
Suppose that $(G,\nu)$ is amenable and computable. 
Let $D \cup \{ n \} \subset_{fin} \mathbb{N}$. 
Applying the enumeration of all finite sets 
we are looking for $F\subset_{fin} \mathbb{N}$ 
which satisfies the conditions of Definition \ref{ca} for $\nu (D)$. 
Since by Remark \ref{rp} there is an algorithm verifying all equalities of the form $\nu(d_k) \nu(f_i)=\nu(f_j)$, 
where $f_i,f_j\in F$ and $d_k\in D$, 
we can algorithmically check if $\nu(F)\in \mathfrak{F} \o l_{G, \nu(D)}(n)$.
Furthermore, verifying all equalities of the form $\nu(f_k)=\nu(f_l)$, where $f_k, f_l \in F$, we can check if $|F|=|\nu(F)|$. 
Since $(G,\nu)$ is amenable we eventually find the required $F$.

(ii)$\implies$(i). 
Our proof is a modification and a simplification of the construction of Theorem 4.1 from \cite{MC3}.
It is clear, that the existence of an algorithm for (ii) implies amenability of $(G,\nu)$. 
Therefore, we only need to show that $(G,\nu)$ is computable.
According to Remark \ref{rp} it suffices to show that there is an algorithm which for any $n_1,n_2 \in \mathbb{N}$ 
verifies if $\nu(n_1) = \nu(n_2)$.

Fix $n_1,n_2$. 
Let $E$ be the set $\{n_{1},n_{2}\}$.
We use the algorithm for (ii) to find a set $F$ corresponding to $5$ and $E$, 
i.e. $\nu(F) \in F \o l_{G, \nu(E)}(5)$ and $|F|=|\nu(F)|$. 
Let $F = \{f_1, f_2,\ldots, f_k\}$. 

For each $i\in \{ 1,2 \}$ we define $\Sigma_i \subseteq \{ (f,f') \, | \, \nu (n_i) \nu (f) = \nu (f')\, , \, f,f'\in F \}$ by the following procedure. 
Having $f,f'\in F$ apply the algorithm of enumeration of the set $\mathsf{Wrd}^{=}_{\nu}$ for verification if 
$\nu (n_i \star f) = \nu(f' )$, $i=1,2$. 
When we get a confirmation of this equality, we extend the corresponding $\Sigma_i$ by $(f,f')$. 
We apply it simultaneously to each pair $(f,f')$.  
Since 
$$
\forall n\in E \, \,  ( |\nu(F) \cap \nu (n) \nu(F) | \ge \frac{3}{5} |\nu (F) |)  
$$ 
and $|F|=|\nu(F)|$, there is a step of these computations when $\Sigma_1 \cup \Sigma_2$ 
witnesses the inequality above. 
Having this we stop the procedure. 

By the pigion hole principle, there are pairs $(f,f') \in \Sigma_1$ and $(f,f'')\in \Sigma_2$. 
If these pairs are the same, we have $\nu (n_1 )\nu (f) = \nu (n_2 )\nu (f)$, i.e. 
 $\nu (n_1 ) = \nu (n_2 )$.  
If $f' \not=f"$ we have $\nu (f') \not= \nu (f")$, i.e. $\nu (n_1 ) \not= \nu (n_2 )$. 
\end{proof}

The proof of Theorem \ref{ce} gives the following interesting observation. 

\begin{cor} 
Let $(G,\nu)$ be a computably enumerable, amenable group. 
If for some $n\geq 5$ there exists an algorithm, which for every $D\subset_{fin}\mathbb{N}$ with $|D|=2$, finds a set $F\subset_{fin}\mathbb{N}$ such that $\nu(F) \in F \o l_{G, \nu(D)}(n)$ and $|F|=|\nu(F)|$, then $G$ is computable.
\end{cor}
\noindent Using Theorem \ref{ce} we deduce a version of Theorem \ref{re} 
for computable groups. 
This finishes the proof of Theorem \ref{1}. 

\begin{thm}\label{eq}
Let $(G,\nu)$ be a computable group. Then the following conditions are equivalent:

\begin{enumerate}[(i)]
\item $(G,\nu)$ is amenable; 

\item $(G,\nu)$ is computably amenable;

\item there exists an algorithm which, for all pairs $(n,D)$, 
where $n\in \mathbb{N}$ and $D\subset_{fin}\mathbb{N}$, 
finds a finite set $F\subset\mathbb{N}$ such that 
$\nu(F) \in \mathfrak{F}$\o$l_{G,\nu(D)}(n)$ (a weaker version of (ii));

\item $(G,\nu)$ has computable Reiter functions;

\item $(G,\nu)$ has subrecursive F\o lner function. 

\end{enumerate}
\end{thm}

\begin{proof}
By Theorem \ref{ce} we have (i)$\Rightarrow$(ii) and by Lemma \ref{fr}(iv) we have (iv)$\Rightarrow$(iii). Both 
(ii)$\Rightarrow$(iii)$\Rightarrow$(i) and (ii)$\Rightarrow$(v)$\Rightarrow$(i) are easy to see.

It follows that we only need to show that (ii)$\Rightarrow$(iv). 
We start with a finite set $D$ and use an algorithm of (ii) to find a set $F$ corresponding to $2n$. 
Then the characteristic function $\chi_F$ can be taken as $f$ from Definition \ref{cr}. 
Indeed, since the function $\nu$ is injective on $F$ then $\nu_{G ,1}(\chi _{F})$ is the characteristic function of $\nu(F)$, 
which is $\frac{1}{n}$-invariant with respect to $\nu (D)$ by Lemma \ref{fr}(iii).
\end{proof}

\section{Effective F\o lner sequence}

Let $(G,\nu)$ be a computable group. 
By Remark \ref{rp} 
we may assume that the function $\nu$ is injective. 
Therefore, we identify the set $\nu(\mathbb{N})$ with $\mathbb{N}$ and subsets $F$ of $\mathbb{N}$ with $\nu(F)\subseteq G$.

A {\em computable F\o lner sequence} of the group $(G,\nu)$ is a computable sequence $(n_j)_{j\in\mathbb{N}}$ where each $n_j$ is a G\"odel number of some $F_j$, such that $(F_j)_{j\in\mathbb{N}}$ is a F\o lner sequence. 
When $(n_j)_{j\in\mathbb{N}}$ is computable we will also say that $(F_j)_{j\in\mathbb{N}}$ is a computable (or effective) F\o lner sequence.

In this section we consider two questions where F\o lner sequences are involved. 
Firstly we analyse algorithmic complexity of the set of all computable F\o lner sequences. 
The second part of the section concerns the procedure how an invariant mean is built from a F\o lner sequence. 
Having ${\bf x} \in \ell^{\infty}(G)$ the corresponding value of the mean is the limit of an effective sequence of rational numbers.  
We will study the complexity of the corresponding convergence modulus.   
Questions of this kind are quite natural in computability theory, see \cite{kofr} and the discussion in the end of the section. 

\subsection{Algorithmic complexity of computable  F\o lner sequences} 

In the previous section we have shown that amenability of $(G,\nu)$ is equivalent to computable amenability. 
Note that this is also equivalent to existence of an effective F\o lner sequence. 
Indeed, given $\ell$ we put $E= \{ 0,\ldots ,\ell \}$ and using the algorithm for computable amenability, compute the G\"odel number 
$n_{\ell}$ of some $F_{\ell} \in \mathfrak{F} \o l_{G,E}(\ell )$. 
Clearly, the sequence $(F_j)_{j\in\mathbb{N}}$ is a F\o lner's one and the sequence 
$(n_j)_{j\in\mathbb{N}}$ is a computable F\o lner sequence. 

The following theorem classifies the set of all computable F\o lner sequences of the group $(G,\nu)$ 
in the arithmetical hierarchy. 
This is Theorem 2 from Introduction. 

\begin{thm}\label{fs} 
Let $(G,\nu)$ be a computable group. 
The set of all computable F\o lner sequences of $(G,\nu)$ belongs to the class $\Pi^0_3$. 
Moreover, for $G=\bigoplus\limits_{n\in\omega} \mathbb{Z}$ it is a $\Pi^0_3$-complete set.
\end{thm}

\begin{proof} 
Let $\varphi(x,y)$ be a universal recursive function, and $\varphi_x(y)=\varphi(x,y)$ be the recursive function with the number $x$. 
We identify computable F\o lner sequences with numbers of recursive functions which produce these sequences.
The set of these numbers will be denoted by $\mathfrak{F}_{seq}(G)$.

It is straightforward that $m\in \mathfrak{F}_{seq} (G)$ if and only if the following formula holds:
\begin{align}
(\phi(m,y) &\text{ is a total function})\wedge\notag
(\forall x,n \in \mathbb{N})(\exists l)(\forall k,s)\Big(k>l\wedge (\phi(m,k)=s)\\ 
&\wedge(\mbox{$s$ is a G\" odel number of $F$})\rightarrow 
 \, \,  \Big( \frac{|F\setminus x\star F|}{|F|}<\frac{1}{n} \Big) \Big) . \label{seq7} 
\end{align}

Given number $s$ the inequality $ \frac{|F\setminus x\star F|}{|F|}<\frac{1}{n} $ can be verified effectively. 
Since the set of numbers of all total functions belongs to the class $\Sigma_2^0$ (see \cite{sri}), it is easy to see that the set of all $m$ which satisfy (\ref{seq7}) is a $\Pi_3^0$ set. 
This proves the first part of the theorem.

We remind the reader that $W_t=\mathsf{Dom}\varphi_t$ is the computably enumerable set with a number $t$.
The set $\overline{Cof}=\{e:\forall n\; W_{\varphi_e(n)} \text{ is finite}\}$, is known to be a $\Pi^0_3$-complete set (\cite{sri}, p. 87). 
To prove the second part of the theorem, assume that $G=\bigoplus\limits_{n\in\omega} \mathbb{Z}$. Let us show that the set $\overline{Cof}$ is reducible to $\mathfrak{F}_{seq}(G)$. 

We present $\bigoplus\limits_{n\in\omega} \mathbb{Z}$ as $\bigoplus\limits_{n\in\omega} \langle g_n\rangle$. We shall construct a sequence $\{F_s^e \, | \, e,s \in \mathbb{N} \}$ such that $e\in\overline{Cof}$ iff $\{F_s^e\, | \, s\in \mathbb{N} \}$ is a F\o lner sequence. 
For each $e$ let us fix a computable enumeration of the set $\{(n,x): x\in W_{\varphi_e(n)}\}$.
We can assume that this enumeration is without repetitions. 

For a given $s$, we use the enumeration of the set $\{(n,x): x\in W_{\varphi_e(n)}\}$ to find the element $(n_s,x_s )$ with the number $s$. For each $i=1,\ldots, s$ such that $i\neq n_s$ let $F_{s,i}=\{ 0, g^{\pm 1}_{i},g_{i}^{\pm 2}, \ldots ,g_{i}^{\pm s}\}$. 
For $i=n_s$ we put $F_{s,i}=\{ 0, g^{\pm 1}_{i}\}$. 
Then in the former case  $F_s^e$ is an $\frac{1}{2s}$-F\o lner set with respect to $g^{\pm 1}_i$, and in the latter case $F_s^e$ is not a $\frac{1}{4}$-F\o lner set with respect to $g_i$. 
Let $F_s^e=\bigoplus\limits_{1}^{s} F_{s,i}$. 
This ends the construction.
Note that it depends on $e$, and is realized by a Turing machime; the latter can be explicitly found. 

\noindent \textit{Case 1.} $e\notin\overline{Cof}$. 
There exists $n'$ such that $W_{\varphi_{e}(n')}$ is an infinite set. 
Therefore, there exist an increasing sequence $\{s_i\}$ and the number $i'$ such that for all $i>i'$, $F_{s_i}^e$ is not 
an $\frac{1}{4}$-F\o lner set with respect to $g_{n'}$. 
Clearly, the number of the algorithm producing the sequence $\{F_s^e \, | \, s\in \mathbb{N} \}$ does not belong to the set of numbers of F\o lner sequences.

\noindent \textit{Case 2.} $e\in\overline{Cof}$. For all $n$, $W_{\varphi_{e(n)}}$ is a finite set. 
Therefore, for all $n$, there exists the number $s'$ such that for all $s>s'$, $F_{s}^e$ is an $\frac{1}{2s}$-F\o lner set with respect to 
$g^{\pm 1}_{n}$.  
Thus, by an easy argument we see that it is a F\o lner sequence.

Since for every $e$ the number of the algorithm producing $\{F^e_s\}$ can be effectively found, it follows that the set $\overline{Cof}$ is reducible to $\mathfrak{F}_{seq}(G)$, which completes the proof. 
\end{proof}

\subsection{Complexity of convergence moduli} 

Suppose that $G$ satisfies the F\o lner conditions. 
Then $G$ admitsa a left F\o lner net $( F_j  \, | \,  j\in J )$. 
Consider, for each $j\in J$, the mean with a finite
support $m_j : \ell^{\infty} (G) \to \mathbb{R}$, defined by
$$
m_j ( {\bf x} ) = \frac{1}{| F_j |} \sum \{   {\bf x}  (h) \,  | \, h \in F_j \} 
 \mbox{ for all } {\bf x} \in \ell^{\infty} (G) . 
$$  
Then for every $g\in G$ and for every ${\bf x} \in \ell^{\infty}(G)$ we have 
$\lim_j ( g m_j - m_j )( {\bf x} ) = 0$ (see the proof of Theorem 4.9.2 in \cite{csc}). 
Taking a subnet if necessary we may assume that $( m_j \, | \, j\in J)$ converges (in the weak$^*$ topology) to an invariant mean $m$ 
(see Lemma 4.5.9 and Theorem 4.9.2 in \cite{csc}). 

Assume that $(G,\nu)$ is a computable group with an injective $\nu$. 
Let $(n_j)_{j\in\mathbb{N}}$ define an effective F\o lner sequence of the group $(G,\nu)$, where each 
$n_j$ is a G\"odel number of some $F_j$, such that $(F_j)_{j\in\mathbb{N}}$ is a F\o lner sequence. 
The corresponding mean $m$ can be viewed as a measure $2^G \, \to \, [ 0,1 ]$. 
Note that if ${\bf x} \in 2^G$ is computable (with respect to $\nu$), then the sequence $(m_j ({\bf x}) \, | \, j\in J )$ is computable and converges to $m ({\bf x})$.  
The question which we study in this section concerns moduli of this convergence. 

\begin{rem} 
Having a computable F\o lner sequence $\mathcal{F}$ and a computable ${\bf x} \in 2^G$ let us define: 
$$ 
Mod_{\mathcal{F}} ({\bf x}) =  \{ (k,j) \, | \, \forall j' , j^" ( j^" > j< j' \to |m_{j'} ({\bf x}) - m_{j^"} ({\bf x}) | < \frac{1}{k} \} .  
$$ 
It is easy to see that the complement of $Mod_{\mathcal{F}} ({\bf x})$ in $\mathbb{N} \times \mathbb{N}$ is computably enumerable.   
Furthermore, when  $Mod_{\mathcal{F}} ({\bf x})$ is computably enumerable, the function 
$$ 
mod_{\mathcal{F},{\bf x}} : k \to \mathsf{min} \{ j \, | \, (k,j) \in  Mod_{\mathcal{F}} ({\bf x}) \} 
$$ 
is computable. 
\end{rem} 
The main question which we study below concerns the growth of $mod_{\mathcal{F},{\bf x}}$ introduced in this remark. 
In order to simplify the situation assume that $m ({\bf x})$ is a given rational number. 
In fact, we will consider the "pointed" set  
$$ 
Mod^p_{\mathcal{F}} ({\bf x}) =  \{ (k,j) \, | \, \forall j' ( j< j' \to |m_{j'} ({\bf x}) - m ({\bf x}) | < \frac{1}{k} \} .  
$$
As above, the complement of $Mod^p_{\mathcal{F}} ({\bf x})$ in $\mathbb{N} \times \mathbb{N}$ is computably enumerable.   

We will consider the group $\mathbb{Z}$ (under some standard 1-1 enumeration) and the effective  F\o lner family
$$
\mathcal{F} = (\{ - i , -i +1 , \ldots , 0 , \ldots , i-1 , i \} \, | \, i \in \mathbb{N} ). 
$$    
(As we already noted in Section 2,  $\{ - i , -i +1 , \ldots , 0 , \ldots , i-1 , i \}$ is 
$\frac{1}{2i}$-F\o lner with respect to the generator of $\mathbb{Z}$.)  
Our main result below shows that the growth of $mod_{\mathcal{F},{\bf x}}$ is not bounded by a primitive recursive function. 
We believe that our main construction can be adapted to many other computable groups. 
The construction is presented in the proof of the following theorem (Theorem 3 from Introduction). 

\begin{thm} \label{conv-mod}
Let $f$ be a total computable function $\mathbb{N} \to \mathbb{N}$.  
Then there is a computable ${\bf x}_0 \in 2^{\mathbb{Z}}$ such that with respect to the  
computable F\o lner sequence $\mathcal{F}$ the values  $m_i ({\bf x}_0)$, $i\in \mathbb{N}$, converge to $0$, and for every $k\in \mathbb{N}$ there is $j > f(k)$ such that $| m_j ({\bf x}_0) | \ge \frac{1}{k}$.   
\end{thm} 

\begin{proof} 
We define ${\bf x}_0 \in 2^{\mathbb{Z}}$ such that ${\bf x}_0 (0) = 1$ and ${\bf x}_0 (i) = {\bf x}(-i)$.  
Further details are given in the inductive procedure below. 

At inductive step $k$ assume that at step $k-1> 1$ we have already defined some number $i_{k-1}$ and values 
${\bf x}_0 (i)$ for all $i$ with $-i_{k-1} \le i  \le i_{k-1}$.  
We start at $k=3$, where it is assumed that $i_{2} = 5$ , ${\bf x}_0 (1) = {\bf x}_0 (2) ={\bf x}_0 (3) = 0$ and 
${\bf x}_0 (4) = {\bf x}_0 (5) =1$.  
At every step $k>2$ we also assume that 
$$
\frac{\sum\{ {\bf x}_0 (i) | -i_{k-1} \le i \le i_{k-1} \} } {2 i_{k-1}+1} < \frac{1}{ k-1} \le \frac{ \sum\{ {\bf x}_0 (i) | -i_{k-1} \le i \le i_{k-1} \} } {  2 i_{k-1}-1 }  . 
$$ 
It is easy to see from our description of step 3 that this condition is satisfied for $k - 1=2$.  
Note that it implies 
$$
\frac{1}{k} < \frac{\sum\{ {\bf x}_0 (i) | -i_{k-1} \le i \le i_{k-1} \} } {2 i_{k-1}+1}  . 
$$
Find $i'_{k-1}\ge i_{k-1}$ such that 
$$
\frac{1}{k} \le \frac{2f(k) + \sum\{ {\bf x}_0 (i) | -i_{k-1} \le i \le i_{k-1} \} } {  2f(k) +2 i'_{k-1} +1} < \frac{1}{ k-1}   . 
$$
Let $t_k$ be the minimal natural number $> i'_k$ such that 
$$
 \frac{2f(k) + \sum\{ {\bf x}_0(i) | -i_{k-1} \le i \le i_{k-1} \} } { 2 f(k) +2 t_k +1} < \frac{1}{ k}   . 
$$
Put $i_k = f(k) + t_k$ and define 
\[ 
{\bf x}_0  (j) = \left\{\begin{array}{l@{\quad}r}
0  \quad\mbox{if}\quad \, i_{k-1} < j\le  i'_{k-1} \\
1 \quad\mbox{if}\quad i'_{k-1} < j \le i'_{k-1} + f(k)\\
0  \quad\mbox{if}\quad \, \, i'_{k-1} + f(k) < j \le i_k .
\end{array}\right.
\]
Note that $i_k$ and the values ${\bf x}_0 (i)$, $i\le i_k$, are found in an effective way.  
In particular $I = \{ i_k  \, | \, k\in \mathbb{N} \}$ is a computable set and the sequence  
$$
m_j ({\bf x}_0 ) =\frac{\sum\{ {\bf x}_0 (i) | -j \le i \le j \} } {2 j+1}  \, ,  \, j \in \mathbb{N} , 
$$ 
is computable. 
Furthermore, by the choice of $i_k$ we see that $i_{k} -1 \ge f(k)$ and $m_{i_k -1} ({\bf x}_0 ) \ge \frac{1}{k}$. 

To see that $m_j ({\bf x}_0 ) \to 0$  note that  for every $k$ and every $j$ with $i_{k-1} \le j \le i_k$ the inequality 
$m_j ({\bf x}_0 ) \le \frac{1}{k-1}$ holds. 
\end{proof} 

\begin{cor} 
There is a computable ${\bf x}_0 \in 2^{\mathbb{Z}}$ such that with respect to the  
computable F\o lner sequence $\mathcal{F}$ the values  $m_i ({\bf x}_0)$, $i\in \mathbb{N}$, converge to $0$, and for every primitive recursive function $f$  there is $k\in \mathbb{N}$ and $j > f(k)$ such that $| m_j ({\bf x}_0) | \ge \frac{1}{k}$.  
\end{cor}
\begin{proof} 
We enumerate all primitive recursive functions $g_1,g_2,\dots $ and define
a function $\hat{g}: \mathbb{N}\rightarrow \mathbb{N}$ by the rule
\[ 
\hat{g}(n)=\overset{n}{\underset{i=1}{\sum}} \overset{n}{\underset{j=1}{\sum}}\, \, g_i(j),\hspace*{2mm} n\in \mathbb{N}.
\]  
Clearly, $\hat{g}$ is computable and for each $k$ there is $n_0$ such that $g_k (n) \leqslant \hat{g}(n)$ for all $n>n_0$. 
Applying theorem above to $\hat{g}$ we obtain the statement of the corollary. 
\end{proof}

\begin{rem} 
Let $\Gamma$ be a finitely generated group, $S \subseteq \Gamma$ a finite symmetric
 generating set. 
According to \cite{ES} a {\em F\o lner sequence in the (colored) Cayley graph} $Cay(\Gamma ,S)$ is a sequence
 $F =\{ F_1, F_2, \ldots )$ of finite spanned subgraphs such that for all natural $r \ge 0$ all but finitely
 many of the $F_n$ are $r$-approximations. 
The latter means that  there exists a subset $W$ of vertices of $F_n$ of size $> (1 - 1/r)| F_n |$ such that 
for any $w\in W$ the $r$-neighborhood of $w$ is rooted isomorphic to the $r$-neighborhood of a vertex of 
the Cayley-graph of $\Gamma$ (as edge labeled graphs).  
The group $\Gamma$ is amenable if $Cay(\Gamma ,S)$ has a F\o lner sequence. 
Note, that the family $\mathcal{F}$ considered in Theorem \ref{conv-mod} is a F\o lner sequence 
in the Cayley graph of $\mathbb{Z}$ with respect to the generators $\pm 1$.  
\end{rem}

\paragraph{Connections with computability over reals} 

The material of this section is connected with the topic of complexiity over reals initiated by Ko and Friedman in \cite{kofr}. 
Consider a countable amenable group $G$ which is computable under some 1-1 numbering. 
In fact, we will assume that $G$ is defined on $\mathbb{N}$. 
Let $\mathcal{F} = ( F_0 , F_1 , \ldots )$ be a computable F\o lner sequence in $G$.   
Consider $G$ as the following disjoint union: 
\[ 
F_0 \dot{\cup} (F_1 \setminus F_0) \dot{\cup} (F_2 \setminus (F_0 \cup F_1 )) \dot{\cup} \ldots \, = 
\{ g_1, g_2 , \ldots \}  
\] 
with $F_i \setminus (F_0 \cup \ldots \cup F_{i-1} )= \{ g_{\ell_{i-1} +1}, \ldots , g_{\ell_i } \}$. 
We identify ${\bf x} \in 2^G$ with 
\[ 
\frac{{\bf x}(g_1)}{2} + \frac{{\bf x}(g_2)}{4} + \ldots + \frac{{\bf x}(g_i )}{2^i} + \ldots \, . 
\]
Now the following function maps $[0,1]$ into $[0,1]$: 
\[ 
m_i ({\bf x} ) = \frac{1}{|F_i |} \sum \{ {\bf x}(h) \, | \, h\in F_i \} . 
\]
Since this function has finitely many values, according to Lemma 2.1 from \cite{kofr} and example (iii) after it, 
it is a partial recursive function from $[0,1]$ to $[0,1]$.  
At this stage we remind the reader that according to Definition 2.1 of \cite{kofr} 
 intuitively the computation of a partial recursive real function $f$ is as follows. 
For a given dyadic number ${\bf x}$ and and a natural number $n$, the Turing machine 
tries to find a dyadic rational number $d_n$ of length $n$ such that $d_n$ is close to $f({\bf x})$ up to $\frac{1}{2^{n}}$. 
During the computation, ${\bf x}$ is used as  an oracle. 
As a result $f({\bf x}) = \lim_n d_n$. 

Let $m({\bf x}) = \lim_i m_i ({\bf x})$. 
This function is a measure representing the invariant mean defined in the first paragraph of this section. 
Is it a partial recursive function? 
The answer is "no" by the following reason. 
By Theorem 2.2 of \cite{kofr} a partial recursive function is continuous on its domain. 
Note that 

$\bullet$ for every interval $(q,q') \subseteq [0,1]$, for every real $\varepsilon >0$ and for every set $X\subseteq G$ that 
is represented by some $r \in (q,q')$ (as an element of $2^G$), there is $X'\subseteq G$ and $r'\in (q,q')$ which represents $X'$ 
such that $|r - r' |  < \varepsilon$ but $|m(r) - m(r')| > \frac{1}{2}$. 

\noindent 
Indeed, assuming that ${\bf x} \in 2^G$ corresponds to $X$ and $r$, choose a number $n$ such that 
\[ 
q< \frac{{\bf x}(g_1)}{2} + \frac{{\bf x}(g_2)}{4} + \ldots + \frac{{\bf x}(g_n )}{2^n} < q' \, \mbox{ and } 
\Big|\frac{{\bf x}(g_{n+1})}{2^{n+1}} + \frac{{\bf x}(g_{n+2})}{2^{n+2}} + \ldots + \frac{{\bf x}(g_i )}{2^i} + \ldots \, \Big| <\varepsilon . 
\]
Then defining ${\bf x}'$ to be ${\bf x}(g_i)$ for $i\le n$ and ${\bf x}' (g_i) =0$  for $i>n$, we obtain $r'$ such that $m(r') = 0$. 
 It can happen that $m(r) \le \frac{1}{2}$. 
In this case we define ${\bf x'}$ so that ${\bf x}'(g_i )=1$ for all $i> n$. 
It represents $r'$ such that $m(r') = 1$.  
We conclude by the following statement. 

\begin{pr} 
The measure $m(r)$ is not a partial recursive function in any interval from $[0,1]$. 
\end{pr} 

It is worth mentioning that Theorem 3.2 of \cite{kofr} describes polynomial time computable real valued functions 
as some special limits of simple piecewise linear functions.  
Convergence moduli naturally appear in the definition of these limits, see Definition 3.8 of that paper. 
This seems slightly analogous to non-computability of $m(r)$ and non-boundedness of convergence moduli in 
$m_i (r) \to m(r)$.

\section{Metric groups and amenability}   

Locally compact groups form the basic area of classical amenability theory. 
Thus analysis of computability aspects of amenability in the case of metric locally compact groups is a natural challenge.   
On the other hand motivated by strong progress made in the last two decates in the general case of amenable topological groups  
(see \cite{GriHarp} \cite{KPT}, \cite{pestov16}, \cite{SchneiderThom}, \cite{thom}), we study the subject without the restriction 
of local compactness.   

In this section a natural framework for computable amenability of metric groups is presented. 
We  will see that some results from the previous sections  have natural counterparts in the metric case. 
On the other hand, we will also describe some new issues compared to the discrete case.    

\subsection{Computable presentations of metric groups} 

We consider computably enumerable metric groups using the standard approach of computability theory.  
It basically corresponds to computable presentations of Polish spaces considered in  \cite{mos} and \cite{W}. 
See also more recent papers \cite{MelMon}, \cite{DGP}, \cite{FM} where some kinds of computable presentations of continuous metric structures is considered. 

We usually assume that a metric group is taken with a right-invariant metric $\le 1$. 
This is a more general case compared to papers mentioned above, where  bi-invariantness of the metric is assumed.  

\begin{df}\label{df0}
Let $(G,d)$ be a metric group and $\nu: \mathbb{N} \rightarrow G$ 
be a function such that $\nu (\mathbb{N})$ is dense in $G$. 
We call the triple $(G,d, \nu)$ a {\em numbered metric group}.
The function $\nu$ is called a {\em numbering} of $(G,d)$.
If $g\in G$ and $\nu(n)=g$, then $n$ is called a number of $g$.
\end{df} 

It is worth noting that  a numbered metric group $(G,d,\nu)$ can be viewed together with 
the additional condition that $\nu (\mathbb{N})$ is a subgroup of $G$ and $\nu$ is a homomorphism from the enumerated group 
$(\mathbb{N}, *, ^{-1},1)$ introduced in Section 2.1 (i.e. a computable copy of $\mathbb{F}_{\omega}$).  
Indeed, the numbering $\nu$ from the definition, naturally extends to a homomorphism 
$$
\mathbb{F}_{\omega} \to \{ w(\nu (n_1),\ldots , \nu (n_s )) \, | \, w(x_1 , \ldots , x_s ) \mbox{ is a group word over } x_1 , \ldots , x_s , \ldots \}  \le G, 
$$ 
which can be viewed as a numbering: 
$$
\nu_1 :  \mathbb{N} \to \{ w(\nu (n_1),\ldots , \nu (n_s )) \,  | \, w(x_1 , \ldots , x_s ) \mbox{ is a group word over } x_1 , \ldots , x_s , \ldots \}   
$$ 
being a homomorphism $(\mathbb{N},*,^{-1},1)\to  G$.     
This explains why  in the definition below we take this assumptions (see also the paragraph after the definition). 

\begin{df} \label{df11} 
\begin{itemize} 
\item Given an enumerated group $(\mathbb{N}, \star , ^{-1}, 1 )$ we call a homomorphism $\nu : \mathbb{N} \to G$ a {\em computably enumerable presentation} of $G$ if $\nu (\mathbb{N})$ is dense in $G$ and the sets 
$$
\mathsf{Wrd}^=_{\nu} := \{(w(n_1  ,\ldots ,n_s), w'(\ell_1 , \ldots ,\ell_t ))\, | \, w(\bar{x}) \mbox{ and }  w'(\bar{y}) \mbox{ are group words and } 
$$ 
$$ 
w(\nu (n_1 ),\ldots ,\nu (n_s)) = w'(\nu (\ell_1 ) , \ldots , \nu (\ell_t )) \mbox{ holds in } G \, \}
$$ 
and 
$$
\mathsf{Wrd}^<_{\nu} := \{(w(\nu (n_1 ) ,\ldots ,\nu (n_s)), w'(\nu(\ell_1 ) , \ldots , \nu (\ell_t )) ,k):\, w(\bar{x}), w'(\bar{y}) \mbox{ are group words and } 
$$ 
$$ 
d(w(\nu (n_1 ),\ldots ,\nu (n_s)), w'(\nu(\ell_1 ) , \ldots , \nu (\ell_t ))) < \frac{1}{k}\}
$$ 
are  computably enumerable.
\item If $(G,d)$ has a computably enumerable  presentation then we say that $(G,d)$ is {\em computably enumerable}. 
\item If a homomorphism $\nu : (\mathbb{N},\star , ^{-1},1 )  \to G$ is a computably enumerable presentation of $G$ and  the sets
$\mathsf{Wrd}^=_{\nu}$ and $\mathsf{Wrd}^<_{\nu}$ 
are computable, then we say that $\nu$ is a {\em computable presentation} and the group $(G,d,\nu )$ is {\em computable}.
\end{itemize} 
\end{df} 

Note that if $(G,d,\nu )$ is a numbered metric group such that the sets $\mathsf{Wrd}^=_{\nu}$ and $\mathsf{Wrd}^{<}_{\nu}$ are computably enumerable, then the  group $(G,d,\nu_1 )$ built after Definition \ref{df0} is computably enumerable too.
Indeed, since any word $u(\nu_1 (m_1),\ldots , \nu_1 (m_s ))$  coincides with some $w(\nu (n_1 ), \ldots , \nu (n_s ))$ 
and the latter one can be found in a computable way, any enumeration of the sets $\mathsf{Wrd}^=_{\nu}$ and $\mathsf{Wrd}^{<}_{\nu}$ with respect to $\nu$ effectively determines an enumeration of the corresponding set with respect to $\nu_1$.

\begin{rem} \label{comp_rec}
Assume that a numbered metric group $(G,d,\nu)$ is defined for an enumerated group $(\mathbb{N}, \star , ^{-1}, 1 )$ 
with computable operations and $\nu$ is a group homomorphism.  
Then the group is computable if and only if the set 
$$
\mathsf{Wrd}^{\le}_{\nu} := \{(w(\nu (n_1 ) ,\ldots ,\nu (n_s)), w'(\nu(\ell_1 ) , \ldots , \nu (\ell_t )) ,k):\, k \in \mathbb{N} \cup \{ \infty \} \, \, , \,   w(\bar{x}), w'(\bar{y}) \mbox{ are group words and } 
$$ 
$$ 
d(w(\nu (n_1 ),\ldots ,\nu (n_s)), w'(\nu(\ell_1 ) , \ldots , \nu (\ell_t ))) \le \frac{1}{k}\}
$$ 
is computable (we identify $0 = \frac{1}{\infty}$). 
Furthermore, it is easy to see that 
\begin{quote} 
the group is computable if and only if it is computably enumerable and there is an algorithm which for any 
$i,j\in \mathbb{N}$ and $\varepsilon \in \mathbb{Q}^+$ 
finds a rational number $\ell$ such that $d(\nu (i), \nu (j)) \in [\ell , \ell +\varepsilon )$.    
\end{quote} 
This shows that $d(\nu (i), \nu (j))$ is a computable real number. 
\end{rem} 

\begin{rem} 
As in Section 2.1 (see Lemma \ref{comp_gr}) we may assume that when $(G,d,\nu)$ is computably enumerable, the operations of the enumerated group  $(\mathbb{N}, \star , ^{-1}, 1 )$ are computable (in fact, we just take the free group $(\mathbb{N}, * , ^{-1}, 1 )$).  
Note that under this assumption the set 
$$ 
\mathsf{MultT} := \{(i,j,k) \, | \, \nu(i)\nu(j)=\nu(k)\}
$$ 
is computable exactly when 
the set $\mathsf{T}_=:= \{(i,j) \, | \,  \nu(i) = \nu(j)\}$ is computable. 
\end{rem} 

\begin{quote} 
$\bullet$ {\em From now on we will consider computably enumerable groups under such presentations that $\nu$ is a homomorphism from 
an enumerated group $(\mathbb{N}, \star , ^{-1}, 1 )$ with computable operations and the image of the numbering is a dense subgroup. }
\end{quote} 

\begin{rem}\label{comp_m-gr}
It is worth noting that when a numbered metric group $(G,d,\nu)$ is computably enumerable and the set 
$\mathsf{T}_=$ 
is computable,  the numbering $\nu$ in the statement above can be chosen to be injective.   
\end{rem}

\begin{rem} 
Assume that $G$ is a conutable discrete group.  
Let us consider it with respect to the $\{ 0,1\}$-metric $d$.  
The following statements are easy. 
\begin{itemize} 
\item If $(G,\nu )$ is computably enumerable in the sense of Section 2.1, then the metric group $(G,d,\nu )$ is computably enumerable. 
\item If $(G,\nu )$ is computable in the sense of Section 2.1, then $(G,d,\nu )$ is a computable presentation in the sense of Definition \ref{df11}. 
\end{itemize} 
\end{rem}

\subsection{Amenability and effective amenability} 
In order to develope computable amenability in the metric case, we apply amenability theory of topological groups
developed by F.M. Schneider and A. Thom in \cite{SchneiderThom}. 
A rough presentation of it is as follows.  

Let $G$ be a topological group, $F_1 ,F_2\subset G$ are finite and $U$ be an identity neighbourhood. 
Let $R_U$ be a binary relation defined as follows: 
$$
R_U=\{ (x,y) \in F_1\times F_2 : yx^{-1} \in U \} . 
$$ 
This relation defines a bipartite graph on $(F_1 , F_2 )$, say $\Gamma$. 
A matching in $\Gamma$ is an injective map  $\phi : D \to F_2$ such that $D\subseteq F_1$ and $(x,\phi (x)) \in R_U$ for all $x \in  D$. 
A matching $\phi$ in $\Gamma$ is said to  be perfect if $dom(\phi ) = F_1$. 
Furthermore, the {\em  matching number} of $\Gamma$ is defined to be
$$
\mu (\Gamma ) = \mathsf{sup} (| \mathsf{dom}(\phi )| | \phi \mbox{ matching in }\Gamma \}. 
$$ 
By Hall's matching theorem this value 
is computed as follows: 
$$
\mu (F_1 ,F_2 ,U) = |F_1 | - \mathsf{sup} \{ |S| - |N_R (S)| : S\subseteq F_1 \} , 
 $$ 
where $N_R (S) = \{ y\in F_2 : (\exists x\in S) (x,y)\in R_U \}$. 

Theorem 4.5 of \cite{SchneiderThom} 
gives the following description of amenable topological groups. 
\begin{quote} 
{\em Let $G$ be a Hausdorff topological group. The following are equivalent. \\ 
(1) $G$ is amenable. \\
(2) For every $\theta \in (0,1)$, every finite subset $D\subseteq G$, and every identity neighbourhood $U$, 
there is a finite non-empty subset $F\subseteq G$ such that 
$$
\forall g\in D (\mu (F,gF,U) \ge \theta |F|). 
$$ 
(3) There exists $\theta \in (0,1)$ such that for every finite subset $D\subseteq G$, and every identity neighbourhood $U$, 
there is a finite non-empty subset $F\subseteq G$ such that 
$$
\forall g\in D (\mu (F,gF,U) \ge \theta |F|). 
$$
}
\end{quote} 
It is worth noting here that when an open neighbourhood $V$ contains $U$ 
the number  $\mu (F,gF,U)$ does not exceed $\mu (F,gF,V)$. 
In particular, in the formulation above we may consider neighbourhoods $U$ 
from a fixed base of identity neighbourhoods. 
For example in the case of a right-invariant metric group $(G, d)$ 
we may take all $U$ in the form of metric balls 
$B_{<q} = \{ x : d(1,x) < q\}$, $q\in \mathbb{Q}\cap (0,1)$.  
It is also clear that we can restrict all $\theta$ by rational ones. 
From now on we work in this case. 

Let $B_q = \{ x : d(1,x) \le q\}$, $q\in \mathbb{Q}\cap (0,1)$.  
Notice that the corresponding versions of statement (2) above 
are equivalent for $U$ of the form $B_{<q}$ and of the form $B_q$. 
Indeed, this follows from the observation that 
 $\mu (F,gF,B_{<q})\le \mu (F,gF,B_q )$ and 
  $\mu (F,gF,B_q )\le \mu (F,gF,B_{<r})$ for $q<r$. 

The following observation is an important point used in our approach. 
\begin{quote} 
$\bullet$ 
It is well-known that a topological group is amenable if and only if it contains a dense amenable subgroup. 
Furthermore, then every dense subgroup is amenable. 
In particular, the formulation of the Schneider-Thom theorem given above still holds if in conditions (2) and (3) 
the group $G$ is replaced by a fixed dense subgroup.  
\end{quote} 
\begin{rem} 
When $G$ is a contable group considered with respect to the $\{ 0,1 \}$-metric, 
the set $U$ appearing in (2) and (3) can be taken $U = \{ e \}$. 
Note that in this case the number $\mu (F_1 ,F_2 ,U)$ is just $|F_1 \cap F_2 |$. 
In particular, taking $\theta = \frac{n-1}{n}$ the condition that the subset $F \subset G$ satisfies 
$$
\forall g\in D (\mu (F,gF,U) \ge \theta |F|)  
$$
just means that $F$ is $\frac{1}{n}$-F\o lner with respect to $D$.  
\end{rem} 
This remark suggests calling $F$ with 
$$
\forall g\in D (\mu (F,gF,B_{<q}) \ge \frac{n-1}{n} |F|)  
$$
to be $\frac{1}{n}$-{\em F\o lner} with respect to $D$ and $q$.  
Now we will say that a sequence $(F_j)_{j\in \mathbb{N}}$ 
of non-empty finite subsets of $G$ is a {\em F\o lner sequence} 
if for every $g \in G$ the following condition holds:
\begin{equation}
 \lim\limits_{n\rightarrow \infty} \mu (F_n,gF_n,B_{<\frac{1}{n}}) = 1 .
\end{equation} 
It is easy to see that existence of F\o lner sets for all $n$ and $D$  
is equivalent to existence of a F\o lner sequence. 
In fact, this is the theorem of Schneider and Thom stated above. 
It can be viewed as  {\em the metric version of F\o lner condition of amenability}.

\bigskip 

We can now formalize computable amenability of numbered metric groups. 
Assume that $(G,d, \nu)$ is a numbered, right-invariant metric group such that $\nu (\mathbb{N})$ is a dense subgroup of $G$. 
The situation that $G=\nu (\mathbb{N})$ and $d$ is the $\{ 0,1\}$-metric, is possible.

When $D\subset_{fin} \mathbb{N}$ and $m,n \in \mathbb{N}$, let us denote by  $\mathfrak{F} \o l_{G,\nu(D), m}(n)$ the family of all finite subsets $F \subset \mathbb{N}$ which satisfy 
$$
\forall e\in D \, \,  (\mu (\nu(F),\nu (e) \nu(F),B_{<\frac{1}{m}}(1)) \ge \frac{n-1}{n} |\nu (F) |)  .
$$
From the point of view of the terminology of Sections 2 and 3, it is more natural to include into this family sets $\nu (F)$ instead of $F$. 
However this would be slightly incovenient below.  

In the case of the Schneider-Thom version of F\o lner's condition of amenability, we again consider two types of effectiveness: they correspond to ones from Section 3. 
\begin{df}\label{cema} 
The numbered metric group $(G,d,\nu )$ is $\Sigma$-{\em amenable}, if there is an algorithm which 
for all triples $(m,n,D)$ where $m,n\in \mathbb{N}$ and 
$D\subset_{fin}\mathbb{N}$, finds a set $F\subset_{fin}\mathbb{N}$ having $F' \subseteq F$ with $F'\in \mathfrak{F} \o l_{G,\nu(D), m}(n)$. 
\end{df}
\begin{rem} 
In the case when the group $(G,d,\nu)$ is discrete with the $\{ 0,1\}$-metric we arrive at the formulation that 
it is $\Sigma$-amenable if
there exists an algorithm which for all pairs $(n,D)$, where $n\in \mathbb{N}$ and 
$D\subset_{fin} \mathbb{N}$, finds a set $F\subset_{fin}\mathbb{N}$ with 
$ \nu(F') \in \mathfrak{F} \o l_{G,\nu(D)}(n)$ where $F'\subseteq F$. 
This is exactly Definition \ref{cea}. 
\end{rem}

\begin{df}\label{cma} 
The numbered metric group $(G,d,\nu )$ is {\em computably amenable}, if there is an algorithm which 
for all quadrangles $(\ell ,m,n,D)$ where $\ell ,m,n\in \mathbb{N}$ and 
$D\subset_{fin}\mathbb{N}$, finds a set $F \in \mathfrak{F} \o l_{G,\nu(D), m}(n)$ with 
$|F|=|\nu(F)|$, together with an assignament $(i,j) \to q$ where $i,j\in F$, $q\in \mathbb{Q}^+$ and  
$d(\nu (i), \nu (j)) \in [q, q+ \frac{1}{\ell})$. 
 \end{df} 

\begin{rem} 
It is easy to see that in the discrete case the group $(G,d,\nu)$ is computably amenable if and only if 
it satisfies  Definition \ref{ca}. 
\end{rem}

The discussion after the formulation of the Schneider-Thom theorem implies that  for a numbered metric group, 
$\Sigma$-amenability implies amenability of $G$. 
It is clear that computable amenability 
imples $\Sigma$-amenability. 
Furthermore, $\Sigma$-amenability implies that the F\o lner function is subrecursive.

\subsection{Effective amenability of computable metric groups}

The following observation shows that amenable groups which have good computable presentation have computable F\o lner sets. 

\begin{pr} \label{T=vsCA}
Assume that a computably enumerable metric group $(G,d,\nu )$ has decidable equality relation: the set $\mathsf{T}_{=} = \{ (i,j) \, : \nu (i) = \nu (j) \}$ is computable. 
Then amenability of $G$ implies that $(G,d,\nu )$ has computable F\o lner sets, which means the following property: 
\begin{quote} 
there is an algorithm which for all triples $(m,n,D)$ where $m,n\in \mathbb{N}$ and 
$D\subset_{fin}\mathbb{N}$, finds a set $F\in \mathfrak{F} \o l_{G,\nu(D), m}(n)$. 
\end{quote} 
\end{pr} 

 Before the proof we give several remarks. 

\begin{rem} \label{cf}  
(1) The assumptions of the first sentence of this proposition imply that the set $\mathsf{Wrd}^{=}_{\nu}$ (Definition \ref{df11}) 
is computable.  
To see this use the assumption that the  multiplication $\star$ on $\mathsf{N}$ is a computable function. 

\noindent 
(2) It is clear that when $(G,d,\nu )$ has computable F\o lner sets, it is $\Sigma$-amenable. 
On the other hand computable amenability implies having computable F\o lner sets.  

\noindent
(3) Note that in the discrete case $(G,d,\nu)$ has computable F\o lner sets 
if there exists an algorithm which, for all pairs $(n,D)$, 
where $n\in \mathbb{N}$ and $D\subset_{fin}\mathbb{N}$, 
finds a finite set $F\subset\mathbb{N}$ such that 
$\nu(F) \in \mathfrak{F}$\o$l_{G,\nu(D)}(n)$.
\end{rem}

\bigskip 

\noindent 
\begin{proof} 
Let us fix an enumeration of all quadrangles of the form $(m, n, D, F)$ where $m,n\in \mathbb{N}$ and 
$D, F \subset_{fin}\mathbb{N}$.
The following procedure, denoted below by $\hat{\vartheta} (m,n,D,F)$, determines quadrangles with $F$ satisfying the condition of the proposition.

For an input  $(m,n,D,F)$ let $F_0  \subseteq F\cup D\star F$ be a set representing the $\mathsf{T}_{=}$-classes in $F\cup D\star F$. 
 Let us fix an enumeration of the set 
$$
\mathsf{T}^{<}_{F_0} = \{(n_1,n_2 ,q): \, n_1 , n_2 \in F_0   \, , \, q \in \mathbb{Q} \cap (0,1) \, , \, d(\nu(n_1),\nu(n_2))<q \}. 
$$ 
After the $m$-th step of this enumeration we obtain a set of restrictions on $d$ in $F_0$.  
Having this we verify whether these restrictions enforce the condition 
$$
\forall g\in D \, \,  (\mu (\nu(F_0 \cap F),\nu (g) \nu(F_0\cap F),B_{<\frac{1}{m}}(1)) \ge \frac{n-1}{n} |F_0 \cap F|)  
$$ 
If this is not the case we make the next step. 
Note that the number $\mu (\nu(F_0 \cap F),\nu (g) \nu(F_0 \cap F),B_{<\frac{1}{m}}(1))$ does not decrease. 
Indeed, since the distances are becoming smaller, new matchings can occur, but the matchings which were already found, 
can only increase to larger ones. 
We stop when the above inequality holds. 
Note that if in the numbered metric group $(G,d,\nu )$ the condition 
$$
\forall g\in D \, \,  (\mu (\nu(F),\nu (g) \nu(F),B_{<\frac{1}{m}}(1)) \ge \frac{n-1}{n} |\nu(F) |)  
$$ 
holds, then it would be recognized at some step of the procedure $\hat{\vartheta}(m,n,E,F)$. 
Under the notation above, then the set $F_0 \cap F$ would serve as the corresponding F\o lner set.  

The algorithm for F\o lner sets of $(G,d,\nu )$ looks as follows. 
Having an input $(m,n,D)$ we enumerate all finite $F$ and for each of them start the procedure  $\hat{\vartheta} (m,n,E,F)$. 
By amenability of $(G,d)$ for some $F$ such a procedure would give the result.  
\end{proof}

\bigskip 

The authors think that the statement of Proposition \ref{T=vsCA} can not be strenthenned to the condition of computable amenability. 
However, at the moment we do not  have any counterexample. 
The following theorem is a metric version of Theorem \ref{ce}. 
This is Theorem 4 from Introduction. 

\begin{thm}\label{CAvsC+A}
Let $(G,d,\nu)$ be a computably enumerable metric group. 
The following conditions are equivalent: 
\begin{enumerate}[(i)]
\item $(G,d, \nu)$ is amenable and computable; 
\item $(G,d, \nu)$ is computably amenable (Definition \ref{cma}). 
\end{enumerate}
\end{thm}

\begin{proof}
(i)$\implies$(ii). This follows from Proposition \ref{T=vsCA} and a straightforfard argument using Remark \ref{comp_rec}.  

(ii)$\implies$(i).
Our proof is based on the proof of Theorem \ref{ce} (and thus it is slightly related to the construction of Theorem 4.1 from \cite{MC3}).
It is clear that the existence of an algorithm for (ii) (i.e. from Definition \ref{cma}) implies amenability of $(G,d,\nu)$. 
Therefore we only need to show that $(G,d,\nu)$ is computable.
It suffices to present an agorithm such that for any $n_1,n_2  \in \mathbb{N}$ and $\varepsilon\in \mathbb{Q}^+$, 
it finds a rational number $q_0 \ge 0$ such that $d(\nu(n_1), \nu(n_2) )\in [ q_0,  q_0 +\varepsilon)$.

Fix $n_1,n_2$. 
Let $D$ be the set $\{n_{1},n_{2}\}$. 
We apply the algorithm for (ii)  to some $(\ell, m,n,D)$ where $\frac{2}{\ell} \le \varepsilon$ , $m> \frac{4}{\varepsilon}$ 
and $\frac{3}{5} \le \frac{n-1}{n}$. 
Let $F\subset \mathbb{N}$ be a set which is the output of the algorithm and 
let $\Sigma^{\ell}_{F} = \{ (f,f',q) \, | \, f,f' \in F, \, q \in \mathbb{Q}^+ , d(\nu (f),\nu (f')) \in [q, q+ \frac{1}{\ell}) \}$.  

For each $i\in \{ 1,2 \}$ we define $\Sigma_i \subseteq \{ (f,f') \, | \, d(\nu (n_i) \nu (f),\nu (f'))\le \frac{1}{m} , f,f'\in F \}$ by the following procedure. 
Having $f,f'\in F$ apply the algorithm of enumeration of the set $\mathsf{Wrd}^{\le}_{\nu}$ for verification if 
\[ 
d(\nu (n_i \star f), \nu(f') )< \frac{1}{m} \, , \, i=1,2. 
\]
When we get a confirmation of this inequality, we extend the corresponding $\Sigma_i$ by $(f,f')$. 
We apply it simultaneously to each pair $(f,f')$.  
Since 
$$
\forall g\in D \, \,  (\mu (\nu(F),\nu (g) \nu(F),B_{<\frac{1}{m}}(1)) \ge \frac{n-1}{n} |\nu (F) |)  
$$ 
and $|F|=|\nu(F)|$, there is a step of these computations when $\Sigma_1 \cup \Sigma_2$ 
confirms the existence of matchings witnessing this inequality. 
Having this, we stop the procedure. 

By the choice of $n$ there are pairs $(f,f') \in \Sigma_1$ and $(f,f'')\in \Sigma_2$. 
Let $(f',f'',q)\in \Sigma^{\ell}_F$. 
Then we have that $d(\nu (n_1 )\nu (f), \nu (n_2 )\nu (f)) \in [q-\frac{1}{m}, q+\frac{1}{\ell}+\frac{1}{m})$. 
Since $d$ is right invariant we see 
\[ 
d(\nu (n_1 ), \nu (n_2 )) \in [q-\frac{1}{m}, q+\frac{1}{\ell}+\frac{1}{m}). 
\] 
In particular $q- \frac{1}{m}$ serves as the requred $q_0$. 
\end{proof}

As in the case of Theorem \ref{ce} we have the following interesting observation. 

\begin{cor} 
Let $(G,d,\nu)$ be a computably enumerable, amenable group. 
If for some $n\geq 5$ there exists an algorithm, which for every $\varepsilon >0$ and $D\subset_{fin}\mathbb{N}$ with $|D| =2$,  finds a set $F\subset_{fin}\mathbb{N}$ such that $F \in F \o l_{G, \nu(D),\frac{5}{\varepsilon}}(n)$ and $|F|=|\nu(F)|$, together with an assignment as in the formulation of Definition \ref{cma} for $\ell \ge \frac{2}{\varepsilon}$, then $(G,d, \nu )$ is computable.
\end{cor} 

The following corollary corresponds to Theorem \ref{eq}. 

\begin{cor}\label{meq}
Let $(G,d, \nu)$ be a computable group. 
Then the following conditions are equivalent:

\begin{enumerate}[(i)]
\item $(G,d)$ is amenable; 

\item $(G,d, \nu)$ is computably amenable;

\item $(G,d, \nu)$ has computable F\o lner sets;

\item $(G,d, \nu)$ is $\Sigma$-amenable.  

\end{enumerate}
\end{cor}

\begin{proof}
By Theorem \ref{CAvsC+A} we have (i)$\Rightarrow$(ii) and by Proposition \ref{T=vsCA} we have (iv)$\Rightarrow$(iii). 
The remaining implications are easy to see. 
\end{proof}

Comparing this corollary with Theorem \ref{eq} the reader observes that we do not include here any statement concerning Reiter's functions. 
At the moment the authors do not have any metric version of Section 3.1. 
This looks as a non-trivial task.

\bigskip 

Let us consider {\bf F\o lner sequences in computable metric groups} . 
Let $(G,d,\nu)$ be a computable metric group. 
By Remarks \ref{comp_m-gr} and \ref{comp_rec} we may assume that the function $\nu$ is injective. 
Therefore we identify the set $\nu(\mathbb{N})$ with $\mathbb{N}$ and subsets $F$ of $\mathbb{N}$ with $\nu(F)\subset G$.

As in Section 4 an {\em effective F\o lner sequence} of the group $(G,d,\nu)$ is an effective sequence $(n_j)_{j\in\mathbb{N}}$ where each $n_j$ is a G\"odel number of some $F_j$, such that $(F_j)_{j\in\mathbb{N}}$ is a F\o lner sequence in $\nu (\mathbb{N} )$ 
(i.e. $\mathbb{N}$). 

By Theorem \ref{CAvsC+A},  amenability of $(G,d,\nu)$ is equivalent to computable amenability. 
This is also equivalent to existence of an effective F\o lner sequence. 
Indeed, apply the argument given in beginning of Section 4.1.  

Let $\varphi(x,y)$ be a universal recursive function, and $\varphi_x(y)=\varphi(x,y)$ be the recursive function with the number $x$. 
We identify effective F\o lner sequences with numbers of recursive functions which produce these sequences.
The set of these numbers will be denoted by $\mathfrak{F}_{seq}(G,d,\nu )$.
The description of this set in the arithmetical hierarchy (see Theorem \ref{fs}) 
has the following counterpart. 
\begin{quote} 
Let $(G,d,\nu)$ be a computable group. 
The set of all effective F\o lner sequences of $(G,d,\nu)$ belongs to the class $\Pi^0_3$. 
\end{quote} 

\noindent
Indeed, it is straightforward that $m\in \mathfrak{F}_{seq} (G)$ if and only if the following formula holds:
\begin{align}
(\phi(m,y) &\text{ is a total function})\wedge\notag
(\forall g\in \nu (\mathbb{N}))(\forall n)(\exists l)(\forall k,f)\Big(k>l\wedge (\phi(m,k)=f)\\ 
&\wedge(\mbox{$f$ is a G\" odel number of $F$})\rightarrow 
 \, \,  ( \mu (F,g F,B_{<\frac{1}{n}}(1)) \ge \frac{n-1}{n} | F | \Big) , \label{seq} 
\end{align}

\noindent 
Given number $f$ the inequality $ \mu (F,g F,B_{<\frac{1}{n}}(1)) \ge \frac{n-1}{n} | F | $ can be verified effectively. 
Since the set of numbers of all total functions belongs to the class $\Sigma_2^0$ it is easy to see that the set of all $m$ which satisfy (\ref{seq}) is a $\Pi_3^0$ set.

\begin{flushleft}
\begin{footnotesize} 
Karol Duda, UPV/EHU, 	Leioa, Bizkaia, Spain 

kduda@impan.pl 

\bigskip 
Aleksander Iwanow

Institute of Computer Science, University of Opole, 

ul. Oleska 48, 45 - 052 Opole, Poland 

aleksander.iwanow@uni.opole.pl 
\end{footnotesize}
\end{flushleft}


\begin{thebibliography}{100} 

\bibitem{BK} H. Becker and A. Kechris, {\em The Descriptive Set Theory
of Polish Group Actions}, Cambridge University Press, Cambridge, 1996.  

%
%

\bibitem{MC2} M. Cavaleri, Computability of F\o lner sets, 
Intern. J. Algebra and Comp., 27(2017), 819 - 830. 

\bibitem{MC3} M. Cavaleri, Følner functions and the generic Word Problem for finitely generated amenable groups, 
J. Algebra 511(2018), 388 - 404. 

\bibitem{csc} T. Ceccherini-Silberstein and  M. Coornaert, 
Cellular automata and groups, Springer Monographs in
Mathematics, Springer-Verlag, Berlin, 2010. 

\bibitem{CK} R. Chen and A. Kechris, Structurable equivalence relations, 
Fund. Math. 242 (2018), 109 -- 185. 

%
\bibitem{DGP} F. Didehvar, K. Ghasemloo and M. Pourmahdian, 
 Effectiveness in RPL, with applications to continuous logic, Ann. Pure Appl. Logic 161 (2010), 789 - 799. 
\bibitem{DuI} K. Duda and A. Ivanov,  On decidability of amenability in computable groups,
Arch. Math. Log. 61 (2022), no. 7 - 8, 891 – 902, 
\bibitem{CPD} K. Duda and A. Ivanov, Computable paradoxical decompositions, Intern. J. Algebra and Comp. 32 (2022), 953 - 967. 
\bibitem{ES} G. Elek and E. Szab\'{o}, Sofic representations of amenable groups. Proc. Amer. Math. Soc.  139 (2011), no. 12, 4285 – 4291. 
\bibitem{EG}Yu.L. Ershov and S.S. Goncharov, Elementary theories and their constructive models. in: Handbook of recursive mathematics, Vol. 1, pp. 115 -- 165, 
Stud. Logic Found. Math., 138, North-Holland, Amsterdam, 1998.
\bibitem{EG1} Yu.L. Ershov and S.S. Goncharov, {\em Constructive models}. 
Siberian School of Algebra and Logic. Consultants Bureau, New York, 2000. 
\bibitem{FM} J. Franklin and T. McNicholl, Degrees of and lowness for isometric isomorphism, Journal of Logic and Analysis 12 (2020), 1 - 23. 
\bibitem{GKEL}  I. Goldbring, S. Kunnawalkam Elayavalli and  Y. Lodha, Generic algebraic properties in spaces of enumerated groups, Trans. Amer. Math. Soc.	376 (2023), 6245 -- 6282. 	
\bibitem{GriHarp}  
R. Grigorchuk and P.  de la Harpe,  Amenability and ergodic properties of topological groups: from Bogolyubov onwards,  
in: Groups, graphs and random walks, 
London Math. Soc. Lecture Note Ser., 436, Cambridge Univ. Press, Cambridge, 2017,  215 -- 249 

\bibitem{HKP1} E. Hrushovski, K. Krupi\'{n}ski and A. Pillay, Amenability, connected components, and definable actions,
Selecta Mathematica 28 (2022), art. number 16.

\bibitem{HKP2} E. Hrushovski,  K. Krupi\'{n}ski and A. Pillay, On first order amenability, https://arxiv.org/abs/2004.08306.  

\bibitem{HPP} E. Hrushovski, Ya. Peterzill and  A.  Pillay, Groups, measures, and the NIP,
J. Amer. Math. Soc. 21 (2008), no. 2, 563 -- 596. 

%
\bibitem{kechris} A. Kechris, {\em Classical Descriptive Set Theory},
Springer-Verlag, New York, 1995. 
\bibitem{kechrisN} A. Kechris, Dynamic of non-archimedean Polish groups, The Proceedings of the European Congress of Mathematics, Krakow, pp. 375 – 397, EMS Press, 2012. 
\bibitem{KPT} A.S.  Kechris, V.G. Pestov and S. Todorcevic,  
Fra\"{i}ss\'{e} limits, Ramsey theory, and topological dynamics of automorphism groups,  Geom. Funct. Anal. 15 (2005),  106 -- 189. 
\bibitem{khmi} B. Khoussainov and A. Myasnikov, 
Finitely presented expansions of groups, semigroups, and algebras, 
Trans. Amer. Math. Soc. 366 (2014), 1455 - 1474. 
\bibitem{kofr} Ker-I. Ko and  H. Friedman, 
Computational complexity of real functions
Theoretical Computer Science
20 (1982), 323 -- 352. 
\bibitem{ls} R.C. Lyndon and  P.E. Schupp, {\em Combinatorial Group Theory}, 
Ergebnisse der Mathematik und ihrer Grenzgebiete 89, Springer-Verlag,
Berlin, Heidelberg, New York (1977). 
\bibitem{MU} A. Marks and S. Unger, Baire measurable paradoxical decompositions via matchings, Advances in Math.
289 (2016) 397 -- 410
\bibitem{MelMon}  A. Melnikov and A. Montalb\'{a}n, Computable Polish group actions, 
J. Symb. Logic 83 (2018), 443 --460. 
\bibitem{mor} N. Moriakov, On effective Birkhoff's ergodic theorem for computable actions of amenable groups.
Theory Comput. Syst.  62 (5) (2018), 1269 -- 1287. 
\bibitem{mos} J. Moschovakis, {\em Descriptive set theory}. Second edition,  
Mathematical Surveys and Monographs, 155, AMS, Providence, 2009.  
\bibitem{pat} A.L.T. Paterson, {\em Amenability}, 
Mathematical Surveys and Monographs, vol 29, 1988.
\bibitem{pestov16}  V. Pestov,   
Amenability versus property (T) for non locally compact topological groups,     
Trans. Amer. Math. Soc. 370 (2018), 7417 -- 7436  

\bibitem{SchneiderThom} 
F.M. Schneider and A. Thom, On F\o lner sets in topological groups, 
Compositio Math. 154 (2018),1333--1361.  
\bibitem{thom} F.M. Schneider and  A. Thom, The Liouville property and random walks on topological groups
 Comment. Math. Helv. 95 (2020), 483 - 513. 
\bibitem{sim} S.G Simpson, Medvedev degrees of two-dimensional subshifts of finite type, Ergodic Th.  Dyn. Systems , 34 (2014) , no. 2,  679 - 688.
\bibitem{sri} R.I. Soare, {\em Turing Computability, Theory and Applications}, 
Springer-Verlag Berlin Heidelberg (2016). 
\bibitem{W} K. Weihrauch,  {\em Computable Analysis}. Springer, Berlin, 2000. 
\end{thebibliography}
\end{document}